\documentclass[11pt]{article}
\usepackage[a4paper]{anysize}\marginsize{3.5cm}{3.5cm}{1.3cm}{2cm}
\pdfpagewidth=\paperwidth \pdfpageheight=\paperheight
\usepackage{amsfonts,amssymb,amsthm,amsmath,eucal}
\usepackage{pictexwd,dcpic}
\usepackage{etex}
\usepackage{pgf,eucal}
\usepackage{bbm,array}
\usepackage{graphicx}
\usepackage{tikz} 
\usepackage{float}
\restylefloat{table}
\usepackage{subfigure}
\usepackage{caption}

\usetikzlibrary{arrows}
\usepackage{enumitem}
\setlist[itemize,3]{leftmargin=-.5in}

\theoremstyle{plain}
\newtheorem{thm}{Theorem}[section]
\newtheorem{theorem}[thm]{Theorem}

\newtheorem{lemma}[thm]{Lemma}
\newtheorem{proposition}[thm]{Proposition}
\newtheorem{corollary}[thm]{Corollary}

\newcommand\beginproof[1]{\trivlist\item[\hskip\labelsep{\em #1.}]}
\renewcommand\proof{\beginproof{Proof}}
\newcommand\proofof[1]{\beginproof{Proof of #1}}
\def\endproof{\hspace*{\fill}\endproofsymbol\endtrivlist}

\def\endproofsymbol{\frame{\rule[0pt]{0pt}{6pt}\rule[0pt]{6pt}{0pt}}}

\theoremstyle{definition}
\newtheorem{definition}[thm]{Definition}
\newtheorem{remark}[thm]{Remark}
\newtheorem{example}[thm]{Example}

\newtheorem{thevarthm}[thm]{\varthmname}

\newenvironment{varthm*}[1]{\trivlist\item[]{\bf #1.}\it}{\endtrivlist}

\renewcommand\geq{\geqslant}
\renewcommand\le{\leqslant}
\renewcommand\leq{\leqslant}

\newcommand\be{\begin{eqnarray*}}
\newcommand\ee{\end{eqnarray*}}

\renewcommand\P{\mathbb P}

\newcommand\newop[2]{\def#1{\mathop{\rm #2}\nolimits}}
\newop\log{log}
\newop\ord{ord}
\newop\Gal{Gal}
\newop\SL{SL}
\newop\GL{GL}
\newop\Bl{Bl}
\newop\mult{mult}
\newop\mass{mass}
\newop\div{div}
\newop\codim{codim}
\newop\sing{sing}
\newop\vdim{vdim}
\newop\edim{edim}
\newop\Ass{Ass}
\newop\size{size}
\newop\reg{reg}
\newop\areg{areg}
\newop\asreg{asreg}
\newop\satdeg{satdeg}
\newop\supp{supp}
\newop\gin{gin}
\newop\ini{in}
\newop\vol{vol}
\newop\sat{sat}
\newop\length{length}
\newop\depth{depth}
\newop\characteristic{char}
\newop\ass{Ass}
\newop\Sh{Sh}

\def\keywordname{{\bfseries Keywords}}%
\def\keywords#1{\par\addvspace\medskipamount{\rightskip=0pt plus1cm
\def\and{\ifhmode\unskip\nobreak\fi\ $\cdot$
}\noindent\keywordname\enspace\ignorespaces#1\par}}
\def\subclassname{{\bfseries Mathematics Subject Classification
(2000)}\enspace}
\def\subclass#1{\par\addvspace\medskipamount{\rightskip=0pt plus1cm
\def\and{\ifhmode\unskip\nobreak\fi\ $\cdot$
}\noindent\subclassname\ignorespaces#1\par}}

\definecolor{qqqqff}{rgb}{0,0,0}
\definecolor{uuuuuu}{rgb}{0,0,0}
\definecolor{zzttqq}{rgb}{0,0,0}
\definecolor{xdxdff}{rgb}{0,0,0}

\begin{document}
\title{On the containment hierarchy for simplicial ideals}
\author{Lampa-Baczy\'{n}ska Magdalena, Malara Grzegorz}

\date{\today}
\maketitle
\thispagestyle{empty}

\begin{abstract}
The purpose of this note is to study containment relations and asymptotic invariants for ideals of fixed codimension skeletons (simplicial ideals) determined by arrangements of $n+1$ general hyperplanes in the $n$--dimensional projective space over an arbitrary field.
\end{abstract}

\section{Introduction}
   The last few years have seen a number of exciting developments at the intersection
   of commutative algebra, algebraic geometry and combinatorics. A number of new methods
   and intriguing asymptotic invariants have evolved out of groundbreaking papers
   by Ein, Lazarsfeld and Smith \cite{ELS01} in characteristic zero and Hochster
   and Huneke \cite{HoHu02} in arbitrary characteristic. In particular, star configurations
   have emerged as a natural testing ground for numerous conjectures related to
   the containment relations between symbolic and ordinary powers of ideals, see
   \cite{GHM13} for a very nice introduction to this circle of ideas. In the present
   note we study a special case of star configurations, namely simplices $\Delta(n)$ cut out
   by $n+1$ coordinate hyperplanes $H_0,\ldots,H_n$ in $\P^n$, or more precisely simplicial complexes
   arising by intersecting all possible tuples of these hypersurfaces. Thus
   a codimension $c$ face $F(i_1,\ldots,i_c)$ is the intersection $H_{i_1}\cap\ldots\cap H_{i_c}$
   for distinct $i_1,\ldots,i_c\in\left\{0,\ldots,n\right\}$. Let $I(n,c)$ denote the ideal
   of the union of all codimension $c$ faces of $\Delta(n)$. We call such ideals \emph{simplicial}.
   There are obvious containments
   \begin{equation}\label{eq: basic inclusions}
      I(n,1)\subset I(n,2)\subset\ldots\subset I(n,n-1)\subset I(n,n).
   \end{equation}
   The containment problem for symbolic and usual powers of ideals has been intensively
   studied in recent years, see e.g. \cite{BoHa10a}, \cite{BoHa10b}, \cite{HaHu13}.
   The first counterexample to the $I^{(3)}\subset I^2$ containment for an ideal
   of points in $\P^2$ announced in \cite{DST13} has prompted another series of papers
   \cite{BCH14}, \cite{HaSe13}, \cite{CoEmTaHO}. In all these works the authors study containment
   relations of the type $I^{(m)}\subset I^r$ for \emph{ a fixed} homogeneous ideal $I$.
   Along these lines  we obtain the following result for simplicial ideals.
\begin{varthm*}{Theorem A}
\label{Theorem A}
   For $n\geq 1$ and $c\in\left\{1,\ldots,n\right\}$, there is the containment
   \begin{equation*}
     \text{(i)}\qquad I^{(m)}(n,c)\subset I^r(n,c)
   \end{equation*}
   if and only if
  \begin{equation*}
     \text{(ii)}\qquad  r\leq \frac{(n+1)k-p}{n-c+2},
 \end{equation*}
  where $m=kc - p$ and $0 \leq p < c$.
\end{varthm*}
   However our approach is a little bit more general. It is motivated by the hierarchy
   established in \eqref{eq: basic inclusions}. Thus we extend the containment problem
   to inclusion relations between symbolic powers of \emph{various} simplicial ideals.
   Our main result in this direction is the following.
\begin{varthm*}{Theorem B}
\label{Theorem B}
   Let $n$ be a positive integer and let $c,d\in\left\{1,\ldots,n\right\}$. If $c\leq d$ and $s\cdot c\leq m\cdot d$ then there is the containment
   \begin{equation*}
      I^{(m)}(n,c)\subset I^{(s)}(n,d).
   \end{equation*}
\end{varthm*}
   Bocci and Harbourne introduced in \cite{BoHa10b} an interesting
   invariant, the resurgence $\rho(I)$ measuring in effect the
   asymptotic discrepancy between symbolic and ordinary powers of
   a given ideal (see Definition \ref{resurgence}). This is a delicate invariant and the family of
   ideals for which it is known is growing slowly, see e.g. \cite{DHNSST14}.
   Here we expand this knowledge a little bit.
\begin{varthm*}{Theorem C}
\label{Theorem C}
   For a positive integer $n$ and $c\in\left\{1,\ldots,n\right\}$ there is
   \begin{equation*}
      \rho(I(n,c))=\frac{c(n-c+2)}{n+1}.
   \end{equation*}
\end{varthm*}
   Note that the $\geq $ inequality was established in \cite[Theorem 2.4.3 b]{BoHa10a} and the case $I(n,n)$ was computed in \cite[Theorem 2.4.3 a]{BoHa10a}.

\section{Preliminaries}

In this section we recall basic definitions and introduce some notation. We work over an arbitrary field $\mathbb{K}$. Let $S(n)=
\mathbb{K}[x_0,\ldots,x_{n}]$ be the ring of polynomials over $\mathbb{K}$.

\begin{definition}
Let $I \subseteq S(n)$ be a homogeneous ideal and let $m \geq 1$ be a positive integer. The $m$-th symbolic power of $I$ is

\begin{displaymath}
I^{(m)}= S(n) \cap \Big( \bigcap_{Q \in \ass(I)} I^{m}_{Q} \Big),
\end{displaymath}
where the intersection takes place in the field of fractions of $S(n)$.
\end{definition}

Although symbolic powers are defined algebraically, they have a nice geometrical interpretation due to the following result of Zariski and
Nagata (see \cite{Eis},Theorem 3.14  and \cite{SiSu09}, Corollary 2.9):

\begin{theorem}
(Nagata-Zariski)
\label{Na-Za}
\noindent
Let $I\subseteq S(n)$ be a radical ideal and let $V$ be the set of zeroes of $I$. Then $I^{(m)}$ consists of all polynomials vanishing to
order at least $m$ along $V$.
\end{theorem}

It follows immediately from the above theorem that there are inclusions

\begin{displaymath}
I= I^{(1)}\supseteq I^{(2)}\supseteq  I^{(3)} \supseteq \ldots
\end{displaymath}
Of course the same is true for usual powers

\begin{displaymath}
I= I^{1}\supseteq I^{2}\supseteq  I^{3} \supseteq \ldots
\end{displaymath}

It is natural to wonder for what $m$ and $r$ there are inclusions

\begin{equation}
\label{eq:ab}
\text{a})\;\; I^{r}\subseteq I^{(m)} \qquad \text{ and } \qquad \text{b})\;\; I^{(m)} \subseteq I^{r}
\end{equation}
It is easy to see that there is the inclusion in a) if and only if $m \le r$. More generally it follows from Theorem \ref{Na-Za} that there is always the inclusion

\begin{equation}
\label{symb.incl}
\big( I^{(a)} \big)^{b} \subseteq I^{(ab)}
\end{equation}

\noindent
As for b) Ein, Lazarsfeld and Smith in characteristic zero and
Hochster and Huneke in arbitrary characteristic showed that there is always containment for $m \geq n \cdot r$. Of course, in certain cases (e.g. complete intersections ideals)
this bound is not optimal and the problem has to be studied individually in any given case. It might be worth to mention that in fact it is not known if the bound $m\geq nr$ is ever optimal.

Here we study ideals $I(n,c)$ of codimension $c$ skeletons of the simplex spanned by all coordinate points in $\mathbb{P}^{n}$. More
exactly,
if $H_{i}$ is the hyperplane  $\{ x_{i}=0 \}$ for $i= 0, \ldots, n$, then the set of zeroes of $I(n,c)$ is the union of all $c$-fold
intersections $H_{i_1}\cap \ldots \cap H_{i_{c}}$ for mutually distinct indices $i_1, \ldots, i_{c} \in \{ 0, \ldots,n \}$.

In connection with containment b) in (\ref{eq:ab}) Bocci and Harbourne  introduced in \cite{BoHa10a} the following quantity.

\begin{definition}
\label{resurgence}
Let $I\subseteq S(n)$ be a non-trivial (i.e.  $I \neq \langle 0\rangle$ and $I \neq \langle 1\rangle$) homogeneous ideal. The resurgence of $I$ is the real number
\begin{displaymath}
\rho(I):= sup \Big\{ \frac{m}{r}: I^{(m)} \nsubseteq I^{r} \Big\}.
\end{displaymath}
\end{definition}

\noindent This invariant is of interest as it guarantees the containment
\begin{displaymath}
I^{(m)}\subseteq I^{r}
\end{displaymath}
for $\frac{m}{r}> \rho(I)$.

\subsection{Monomial ideals}

We identify monomials $x^a=x^{a_0}\cdot\ldots\cdot x_n^{a_n} \in S(n)$ with vectors $(a_0,\ldots ,a_n)\in\mathbb{R}^{n+1}$ in the usual way.

Our approach to Theorems A, B and C relies heavily on the fact that the ideals $I(n,c)$ are Stanley--Reisner ideals, i.e. ideals generated
by square free monomials. More exactly we have the following fact.

\begin{lemma}
\label{lem:inc_gen}
The ideal $I(n,c)$ is generated by all monomials of the form
\begin{displaymath}
x_{i_1} \cdot \ldots \cdot x_{i_{n+2-c}}
\end{displaymath}
for mutually distinct $i_1, \ldots i_{n+2-c} \in \{0,\ldots, n \}$.
\end{lemma}

\noindent The symbolic powers of simplicial ideals are also monomial ideals.

\begin{proposition}
\label{symbolic generators}
Let $m$ be a positive integer. For a monomial $x^a=x_0^{a_0} \cdot \ldots \cdot x_{n}^{a_{n}}$ the following conditions are equivalent:
 \begin{itemize}
  \item [a)] $x^a \in I^{(m)}(n,c)$;
  \item [b)] for all $c$-tuples of mutually distinct indices $i_1 , \ldots i_{c}  \in \{0, \ldots, n \}$
  \end{itemize}
 \begin{equation}\label{eq:i1ic}
 a_{i_1} + \ldots + a_{i_{c}} \geq m.
 \end{equation}
\end{proposition}
\begin{proof}
Based on (\cite{GHM13}, Theorem 3.1.) we claim that $I^{(m)}(n,c)$ is a monomial ideal. Assume that $x^{a} \in I^{(m)}(n,c)$. This means that on all faces of codimension $c$ any derivative of $x^{a}$ of order $m-1$ is equal to zero which
enforces for all $c$-tuples of exponents the condition (\ref{eq:i1ic}).

On the other hand suppose that  $x^{a}$ is a monomial such that (\ref{eq:i1ic}) holds. Since (\ref{eq:i1ic}) is invariant under permutation of variables, it is enough to check that $x^a$ vanishes to order $\geq m$ along one of the codimension $c$ faces. To this end let $F$ be defined by
equations $x_n=x_{n-1}=\ldots=x_{n-c+1}=0$.
For $j_1=n,j_2=n-1,\ldots,j_c=n-c+1$ the monomial $x^{a}$ is divisible by a monomial of degree greater or equal to $m$ in variables
$x_n,\ldots,x_{n-c+1}$, namely by $x_{n}^{a_{n}}\cdot \ldots \cdot x_{n-c+1}^{a_{n-c+1}}$, hence it is in the $m$-th symbolic power of $I(F)$.
\end{proof}

   We can similarly characterize usual powers of simplicial ideals.
\begin{proposition}\label{prop:usual powers}
   Let $r$ be a positive integer. Let
   $x^a=x_0^{a_0}\cdot\ldots\cdot x_n^{a_n}$ be a monomial in $S(n)$.
   The following conditions are equivalent.
   \begin{itemize}
   \item[a)] $x^a$ is in a minimal set of generators of $I^r(n,c)$;
   \item[b)]
   \subitem i) $\sum_{i=0}^n a_i=(n-c+2)r$;
   \subitem ii) $a_i\leq r$ for all $0\leq i\leq n$.
   \end{itemize}
\end{proposition}

\proof
   The implication from a) to b) is obvious in the view of Lemma \ref{lem:inc_gen}.

   For the reverse implication, we proceed by induction. The case $r=1$ follows
   again from Lemma \ref{lem:inc_gen}. Suppose thus that the statement is proved
   for all integers up to $r-1$. Let $x^a$ be a monomial satisfying b).

   Renumbering the variables if necessary, we may assume that the powers
   $a_0,\ldots,a_n$ are ordered
   \begin{equation}
      a_0\geq a_1\geq\ldots\geq a_{n+1-c}\geq\ldots\geq a_n.
   \end{equation}
   The conditions b.i) and b.ii) imply that $a_{n+1-c}\geq 1$. Hence $x^a$
   is divisible by the generator $g=x_0\cdot\ldots\cdot x_{n+1-c}\in I(n,c)$.
   Let
   $$a_i':=\left\{\begin{array}{ccc}
   a_i-1 & \mbox{ for } & 0\leq i\leq n+1-c\\
   a_i & \mbox{ for } & i\geq n+2-c
   \end{array}\right.$$
   and let $x^{a'}=x_0^{a_0'}\cdot\ldots\cdot a_n^{a_n'}$. Note that $a_0',\ldots,a_n'$ satisfy conditions b.i) and b.ii) with $r$ replaced by $r-1$. Hence the induction assumption implies that $x^{a'}$ is an element of $I^{r-1}(n,c)$ and thus
   $x^a=g\cdot x^{a'}\in I^r(n,c)$. And we are done.
\endproof

In the sequel we need the following result

\begin{corollary}
\label{i+ii}
For a monomial $x^a \in I^r(n,c)$, we have
\begin{equation*}
a_{i_1}+\ldots+a_{i_c}\geq r
\end{equation*}
for all $c$--tuples of mutually
   distinct indices $i_1,\ldots,i_c\in\left\{0,\ldots,n\right\}$.

\end{corollary}

\begin{proof}
Since there is always the inclusion $I^r \subset I^{(r)}$, the assertion follows from Proposition \ref{symbolic generators}.
\end{proof}

\section{Triangle}
\label{triangle}

Before we pass to proving the general statements, we want to examine the cases $n=2$ and $n=3$ in more detail. In these cases we obtain more containment relations, which suggest that there might be even more regularity also in the general case. We hope to come back to this problem in the near future.

In this section we consider the ideal $E=I(2,1)= \langle x_0x_1x_2\rangle$ and $V=I(2,2)= \langle x_0x_1,x_0x_2,x_1x_2 \rangle$ in $\mathbb{P}^2$.
We introduce the following notation $$P_0=[1:0:0],\; P_1=[0:1:0],\; P_2=[0:0:1].$$

During the investigation of ordinary and symbolic powers we observed the following behavior of these ideals.
\begin{lemma}
\label{E CI}
The ideal $E$ is a complete intersection ideal.
\end{lemma}
Since the powers of a complete intersection ideal (\cite{Eis}, p.466) are arithmetically Cohen-Macaulay this implies that for all positive integers $k$ we have $E^{k}=E^{(k)}$.
\begin{proposition}
\label{triangle symbolic}
The ideal  $V^{(m)}$ is generated by monomials of the form $ x_{\sigma(0)}^{m-k} x_{\sigma(1)}^{m-k} x_{\sigma(2)}^k $ for any permutation $\sigma $ of $\{ 1,2,3 \}$ and $k \in \big\{0,1,\ldots, \lfloor \frac{m}{2} \rfloor \big\}$.
\end{proposition}
\begin{proof}
Note either directly from the definition of symbolic power or from Proposition \ref{symbolic generators} that monomials of the form $ x_{\sigma(0)}^{m-k} x_{\sigma(1)}^{m-k} x_{\sigma(2)}^k $ are in $V^{(m)}$. Next, consider any monomial $x^a=x_{\sigma(0)}^ix_{\sigma(1)}^jx_{\sigma(2)}^k \in V^{(m)}$; then $i,j,k$ satisfies condition (\ref{eq:i1ic}) with $c=2$. We are looking for a monomial of the form $ x_{\sigma(0)}^{m-k} x_{\sigma(1)}^{m-k} x_{\sigma(2)}^k $   that devides $x^a$. This is certainly the case for the monomials of the form $x_{\sigma(0)}^{i-k}x_{\sigma(1)}^{j-k}x_{\sigma(2)}^{k}$ for some $k \geq 0$. Without loss of generality we may assume that $0\leq k\leq j\leq i \leq m$. We are asking about the maximal possible value of $k$ for which the condition (\ref{eq:i1ic}) is fulfilled. We obtain
$$m \leq i-k+j-k \leq 2m-2k,$$ or equivalently $k \leq \lfloor \frac{m}{2} \rfloor$.
\end{proof}
\begin{example}
$V^{(2)}=\langle x_0^2 x_1^2, x_0^2 x_2^2, x_1^2 x_2^2, x_0 x_1 x_2 \rangle $.
\end{example}
The following proposition is a special case of Theorem A (when $c = n = 2$). The proof of Theorem A is given in section 5. This result is
also a special case of Theorem 2.18 of \cite{DeJa}.
\begin{proposition}
\label{triangleTheoremA}
For all positive integers $r$ and $m$, $V^{(m)} \subset  V^{r}$ if and only if $ 2r \leq \lceil\frac{3m}{2}\rceil$.
\end{proposition}
\begin{proposition}
\label{triangleContainments}
For all positive integers $m$ we have

\begin{itemize}
  \item [a)] $E^{(m)} \subset V^{(2m)}$,
  \item [b)] $E^{(m+1)} \subset V^{(2)} \cdot E^{(m)} \subset V \cdot E^{(m)}$.
\end{itemize}
\end{proposition}

\begin{proof}
a) From Proposition \ref{triangle symbolic} we get $(x_0 x_1 x_2)^{m}  \in V^{(2m)}$, but $E^{(m)}= \langle(x_0 x_1 x_2)^{m}\rangle $.\\
b) We have $ V^{(2)}=\langle x_0^2 x_1^2, x_0^2 x_2^2, x_1^2 x_2^2, x_0 x_1 x_2 \rangle$ and $ E=\langle   x_0 x_1 x_2 \rangle $. It
is easy to see that then $E \subset V^{(2)}$. That implies $E^{m+1} \subset V^{(2)} \cdot E^{m}$, which by Lemma \ref{E CI} is equivalent to
$E^{(m+1)} \subset V^{(2)} \cdot E^{(m)}$. The inclusion  $V^{(2)} \cdot E^{(m)} \subset V \cdot E^{(m)}$ is obvious.
\end{proof}
From Propositions \ref{triangle symbolic}, \ref{triangleTheoremA} and \ref{triangleContainments} we obtain another  interesting observations about the
containment relations. We visualize them in the following diagram. The direction of arrows symbolizes
inclusion between the ideals.

\begin{figure}
\qquad\qquad\qquad\qquad\qquad
\includegraphics[width=200px]{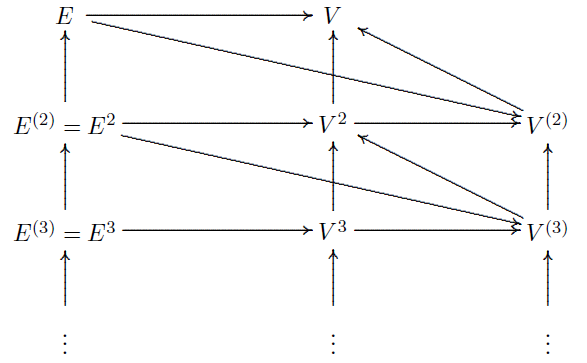}
\end{figure}

\begin{proposition}
\label{power}
For all positive integers $m$ it holds $V^{(2m)}=(V^{(2)})^m$ and $V^{(2m+1)}=V^{(2m)}V$.
\end{proposition}
\begin{proof}
By (\ref{symb.incl}) we have $\big( V^{(2)} \big) ^{m} \subseteq V^{(2m)}$.

On the other hand by Proposition \ref{triangle symbolic} any generator of $V^{(2m)}$ has the form $x_{\sigma(0)}^{2m-j} x_{\sigma(1)}^{2m-j} x_{\sigma(2)}^j$
for some $j\in\{0,1,\ldots,m\}$. Then $$x_{\sigma(0)}^{2m-j} x_{\sigma(1)}^{2m-j} x_{\sigma(2)}^j = (x_{\sigma(0)}^2x_{\sigma(1)}^2)^{m-j}(x_{\sigma(0)} x_{\sigma(1)}x_{\sigma(2)})^j \in (V^{(2)})^m.$$ This ends the first part of the proof of the Proposition. In the second part, using Proposition \ref{triangle symbolic} again, we obtain
\begin{eqnarray*}
V^{(2m)} V&=& V^{(2m)} \cdot \langle x_0
x_1,x_0x_2,x_1x_2\rangle =\\&=& \langle x_0^{2m+1} x_1^{2m+1},x_0^{2m+1} x_2^{2m+1},\ldots,x_1^{m+1} x_2^{m+1}x_0^{m}\rangle = V^{(2m+1)}.
\end{eqnarray*}
\end{proof}

\noindent  Proposition \ref{power} is a special case of Theorem 3.6 of \cite{DeJa}.

\noindent We conclude this section with the following corollary from Proposition \ref{triangle symbolic}
\begin{corollary}
There is the following relation $$V^{(2)}=E+ V^2.$$
\end{corollary}
\begin{proof}
It is a simple observation. Directly from the definition of $V^2$ and from Proposition \ref{triangle symbolic} we get
\begin{eqnarray*}
E+ V^2 & =&\langle x_0^2 x_1^2,
x_0^2 x_2^2, x_1^2 x_2^2, x_0^2 x_1 x_2, x_0 x_1^2 x_2, x_0 x_1 x_2^2, x_0 x_1 x_2 \rangle =\\ &=&\langle x_0^2 x_1^2, x_0^2 x_2^2, x_1^2 x_2^2,
x_0 x_1 x_2 \rangle = V^{(2)}.
\end{eqnarray*}
\end{proof}
\section{Tetrahedron}
\label{tetra}

As a natural generalization of previous results we consider the tetrahedron in $\mathbb{P}^3$.
We write  $$P_0=[1:0:0:0],P_1=[0:1:0:0],P_2=[0:0:1:0],P_3=[0:0:0:1]$$ for the vertices of the tetrahedron.

As before, we denote by $V=I(3,3)=\langle x_0 x_1,x_0 x_2,x_1 x_2,x_2 x_3,x_1 x_3,x_0 x_3\rangle $ the ideal of vertices, $E=I(3,2)=\langle
x_0x_1 x_2,x_1 x_2 x_3,x_0 x_2 x_3,x_0x_1 x_3\rangle $ the ideal of edges and by $F=I(3,1)=\langle x_0 x_1 x_2 x_3 \rangle  $ the ideal of faces.

Once again we take a closer look at the symbolic and ordinary powers of these ideals. We obtain following results.

\begin{lemma}
The ideal $F$ is a complete intersection ideal and this implies the equality $F^{(k)}=F^k$ for all positive integers $k$.
\end{lemma}

\begin{proposition}
\label{tetrahedron symbolic}
For any permutation $\sigma$ of $\{ 0,1,2,3 \}$ the generators of
\begin{itemize}
  \item [1)] $E^{(m)}$ are of the form $ x_{\sigma(0)}^{m-j} x_{\sigma(1)}^{m-j} x_{\sigma(2)}^{m-j} x_{\sigma(3)}^j $ for $j \in
      \big\{0,1,\ldots,\lfloor\frac{m}{2}\rfloor \big\}$,
  \item [2)] $V^{(m)}$ are of the form $ x_{\sigma(0)}^{m-i-j} x_{\sigma(1)}^{m-i-j} x_{\sigma(2)}^i x_{\sigma(3)}^j $ for $ 0 \leq 2i+j \leq  m $ and $ 0 \leq i+2j \leq  m $.
\end{itemize}
\end{proposition}
\begin{proof}
Proof of the case 1) is similar to the proof of Theorem \ref{triangle symbolic} with one additional variable.

 Let focus on the proof of case 2). It is clear that monomials of the given form are in the ideal. To show they generate, consider any
 $x_{\sigma(0)}^{a} x_{\sigma(1)}^{b} x_{\sigma(2)}^{c} x_{\sigma(3)}^{d} $  in the ideal. Without loss of generality we may assume
 $a\leq b \leq c \leq d$ and hence $a+b+c\geq m$. It is easy to see this monomial is divisible by one of the form
 $x_{\sigma(0)}^{i} x_{\sigma(1)}^{j} x_{\sigma(2)}^{k} x_{\sigma(3)}^{l} $ with $i \leq j \leq k =l= m-i -j$. Thus we have
 $i\leq j$ and $i + 2j \leq m$. Thus monomials of this form generate, but these are the same monomilas as given in the part (2) of the statement of Proposition \ref{tetrahedron symbolic}.
 \end{proof}

\begin{proposition}
For all positive integers $m$ it holds that $E^{(2m)}=(E^{(2)})^m$ and $E^{(2m+1)}=E^{(2m)}E$.
\end{proposition}
\begin{proof}
This repeats the reasoning in the proof of Proposition \ref{power} and is left to a motivated reader.
\end{proof}
Immediately from Proposition \ref{tetrahedron symbolic}, Theorem A and Theorem B (we refer to section 5 for the proofs of A and B) we can derive many more containment relations between the ideals $V,E,F$ and their ordinary and symbolic powers. The following diagram presents them when $n = 4$. The direction of arrows symbolizes inclusion between the ideals.

\begin{figure}
\qquad
\includegraphics[width=330px]{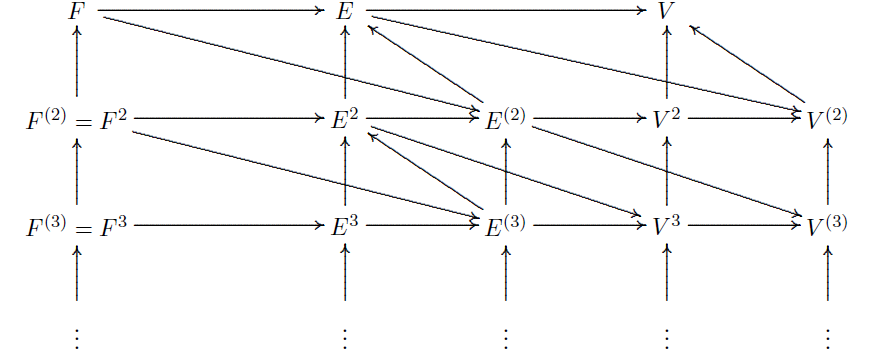}
\end{figure}

As an example we prove here the inclusion $E^{(3)}\subset E^2 $.
\begin{example}
Proposition \ref{tetrahedron symbolic} shows that generators of $E^{(3)}$ are in two forms
$$ x_{\sigma(0)}^3 x_{\sigma(1)}^3 x_{\sigma(2)}^3 = e \cdot x_{\sigma(0)} x_{\sigma(1)} x_{\sigma(2)},$$
and
$$ x_{\sigma(0)}^2 x_{\sigma(1)}^2 x_{\sigma(2)}^2 x_{\sigma(3)} = e \cdot x_{\sigma(3)},$$ where $e= x_{\sigma(0)}^2 x_{\sigma(1)}^2 x_{\sigma(2)}^2 $.
The generators of $E^2$ give us that $e \in E^2$ which finishes the proof.
\end{example}
We finish this section with the following corollary from Proposition \ref{tetrahedron symbolic}.
\begin{corollary}
$E^{(2)} = F+E^2$.
\end{corollary}
\begin{remark}
This generalizes easily to $I^{(2)}(n,2)=I(n,1)+I^2(n,2)$ in the $\mathbb{P}^n$ case.
\end{remark}

\section{General case}
\label{general}

We begin by proving Theorem B. 
\proofof{Theorem B}
Let $x^a=x_0^{a_0}\cdot\ldots\cdot x_n^{a_n} \in I^{(m)}(n,c)$. Fix $j_1,\ldots, j_d$ pairwise distinct in the set $\{0,\ldots,n\}$. By Proposition \ref{symbolic generators} we have
\begin{equation}
\label{trick}
\text{d}\left\{\begin{array}{llllllllcl}
a_{j_1}\;+&&\ldots &  +\; a_{j_c}&& &&&\geq& m\\
&a_{j_2}\;+&&  \ldots  &+\;a_{j_{c+1}}&& &&\geq& m\\

 &&&&\vdots&&&& \\
 &&a_{j_{1+d-c}}\;+&&  \ldots  &&+\;a_{j_{d}} &&\geq& m \\
 a_{j_1}\;&&& +\;a_{j_{2+d-c}} \;+ &&\ldots  &+\;a_{j_{d}}&&\geq& m \\
 &&&&\vdots&&&& \\
a_{j_1}\;+&\; \ldots&+\;a_{j_{c-1}}\;&  && &+\;a_{j_d}&&\geq& m
\end{array} \right.
\end{equation}

Taking into account that every element $a_{j_{i}}$ appears in $c$ inequalities and taking the sum of all of them we get $$ c(a_{j_1} + \ldots +
a_{j_d} )\geq dm.$$
This means that $a_{j_1} + \ldots + a_{j_d} \geq \frac{md}{c}$. Since this holds for arbitrary choice of $j_1,\ldots,j_d$, Proposition \ref{symbolic generators} implies $x^a \in I^{(\lceil \frac{d}{c} m\rceil)}(n,d)$. Since $s \leq \lceil \frac{d}{c} m\rceil$ by assumption, we have $ I^{(\lceil \frac{d}{c} m\rceil)} (n,d) \subset I^{(s)}(n,d)$, so that $x^a \in I^{(s)}(n,d)$ which concludes the proof.
\endproof

One might think that the reverse implication in Theorem B should hold as well, but it is not true in general. For example we have the inclusion $I^{(3)}(3,2) \subset I^{(5)}(3,3)$ and the inequality $s\cdot c \leq m\cdot d$ is obviously false in this case. This follows from Proposition \ref{tetrahedron symbolic} or can be checked directly with a symbolic algebra program.

\proofof{Theorem A}
We show first the implication from ii) to i). Let $x^a \in I^{(m)}(n,c)$. We prove that $x^a \in I^{r}(n,c)$. By Proposition \ref{prop:usual powers} it is enough to show that $a_0+\ldots+a_n \geq r(n-c+2)$.

Without loss of generality we may assume that $a_0\leq a_1\leq\ldots\leq a_n$. From the Proposition \ref{symbolic generators} we obtain
\begin{equation}
\label{proof_Th_A}
a_0+\ldots+a_{c-1} \geq m.
\end{equation}

We claim that $a_{c-1} \geq k$. If not, we get a chain of inequalities $a_0\leq a_1\leq\ldots\leq a_{c-1} \leq k-1$ and with (\ref{proof_Th_A}) we have
$$m \leq a_0+ a_1+\ldots+a_{c-1} \leq c(k-1)=kc-c< kc-p =m.$$
A contradiction.

Thus from $a_{c-1} \geq k$ we derive that
$$ a_0+\ldots+a_{c-1}+a_c+\ldots+a_n \geq m+(n-c+1)k=(n+1)k-p \geq r(n-c+2),$$
where the last inequality is the assumption ii). The claim $x^a \in I^r(n,c)$ follows now from Proposition \ref{prop:usual powers}.

Turning to the proof of the implication from i) to ii) we take $x^a$ such that
\[a=(
\underbrace{k-1,\ldots,k-1}_{\text{$p$}},\underbrace{k,\ldots,k}_{\text{$n-p+1$}})
.\]

It is easy to check that $ x^a \in I^{(m)}(n,c)$ by Proposition \ref{prop:usual powers}. Hence by the assumption i) $x^a \in I^{r}(n,c)$.

The sum of all exponents of $x^a$ is $$ a_0+\ldots+a_n=(n-p+1)k+p(k-1)=nk+k-p=(n+1)k-p.$$

So that by Proposition \ref{prop:usual powers} we obtain $$(n+1)k-p \geq r(n-c+2).$$

This shows ii).
\endproof

This result lets us find the resurgences of ideals $I(n,c)$.

\proofof{Theorem C}
Let $\rho=\rho(I(n,c))$. We show first the inequality $\rho \leq \frac{c(n-c+2)}{n+1}$. To this end let $m$ and $r$ be such that $I^{(m)}(n,c) \nsubseteq I^{r}(n,c)$. Theorem A implies then that
\begin{equation}
\label{res:one}
r > \frac{(n+1)k-p}{n-c+2}.
\end{equation}
From the inequality
\begin{equation*}
m=kc-p < kc - \frac{pc}{n+1} = \frac{c((n+1)k-p)}{n+1},
\end{equation*}
we obtain
\begin{equation}
\label{res:two}
\frac{m(n-c+2)}{(n+1)k-p} < \frac{c(n-c+2)}{n+1}.
\end{equation}

Combining (\ref{res:one}) and (\ref{res:two}) we get

$$\frac{m}{r} < \frac{m(n-c+2)}{(n+1)k-p} < \frac{c(n-c+2)}{n+1}.$$

Passing to the supremum on the left hand side we get $\rho \leq \frac{c(n-c+2)}{n+1}$.

Now, it suffices to show that there exists a particular sequence $(m_k,r_k)$ such that

$$ I^{(m_k)}(n,c) \nsubseteq I^{r_k}(n,c) \;\;\text{ and }\;\; \frac{m_k}{r_k} \underset{k \rightarrow \infty }{\longrightarrow}\frac{c(n-c+2)}{n+1}.$$

We take $m_k=kc$. Then for
$$r_k = \text{min}\Big\{s \in \mathbb{Z} : \;\; s > \frac{(n+1)k}{n-c+2} \Big\} $$
we have $I^{(m_k)}(n,c) \nsubseteq I^{r_k}(n,c)$ by Theorem A. The condition $$\frac{m_k}{r_k} \underset{k \rightarrow \infty }{\longrightarrow}\frac{c(n-c+2)}{n+1}$$ is now obvious.

\endproof

\bigskip
\paragraph*{\emph{Acknowledgement.}}
This research was carried out while the authors were visiting University of Freiburg as ERASMUS students. It is a pleasure to thank Institute of Mathematics in Freiburg for providing nice working conditions and the graduate school Graduiertenkolleg 1821 ``Cohomological Methods in Geometry'' for some additional financial support. We would like to thank Marcin Dumnicki and Tomasz Szemberg for inspiring discussions, helpful remarks and bringing reference \cite{BoHa10b} to our attention. We are greatly obliged to the anonymous referee for an extensive list of very helpful remarks which greatly improved the present note. We are also grateful to the Editor for his patience.

\bigskip \small

\bigskip
   Magdalena~Lampa-Baczy\'nska, Grzegorz~Malara,
   Instytut Matematyki UP,
   Podchor\c a\.zych 2,
   PL-30-084 Krak\'ow, Poland

\nopagebreak

  \textit{E-mail address:} \texttt{lampa.baczynska@wp.pl}

   \textit{E-mail address:} \texttt{gmalara@up.krakow.pl}

\bigskip

   Autors current address:
   Albert-Ludwigs-Universit\"at Freiburg,
   Mathematisches Institut, D-79104 Freiburg, Germany

\end{document}